\documentclass[12pt]{article}
\usepackage[latin1]{inputenc}

\usepackage{amsfonts}
\usepackage{amssymb}
\usepackage{amsthm}
\usepackage{amsmath}
\usepackage{graphicx}
\usepackage{color}
\usepackage{amscd}
\usepackage{multirow}

\newtheorem{theorem}{Theorem}[section]
\newtheorem{proposition}[theorem]{Proposition}
\newtheorem{definition}[theorem]{Definition}

\newtheorem{lemma}[theorem]{Lemma}
\newtheorem{corollary}[theorem]{Corollary}
\newtheorem{remark}[theorem]{Remark}

\numberwithin{equation}{section}

\usepackage{amsfonts}
\usepackage{amsthm}
\usepackage{amsmath}
\usepackage{amscd}
\usepackage{multirow}

\newcommand{\CC}{{\mathbb C}}
\newcommand{\C}{{\mathbb C}}
\newcommand{\cD}{{\mathcal D}}

\newcommand{\cH}{{\mathcal H}}
\newcommand{\cO}{{\mathcal O}}
\newcommand{\cQ}{{\mathcal Q}}
\newcommand{\cM}{{\mathcal M}}
\newcommand{\cF}{{\mathcal F}}
\newcommand{\cE}{{\mathcal E}}

\newcommand{\RR}{{\mathbb R}}
\newcommand{\R}{{\mathbb R}}
\newcommand{\ZZ}{{\mathbb Z}}

\newcommand{\NN}{{\mathbb N}}
\newcommand{\N}{{\mathbb N}}

\newcommand\inw{{{\operatorname{in}}_{\omega}}}

\newcommand\inww{{{\operatorname{in}}_{(-\omega , \omega )}}}

\newcommand\Irr{\operatorname{Irr}}
\renewcommand{\ker}{\operatorname{Ker}}

\begin{document}


\title{Gevrey solutions of  irregular hypergeometric systems in two
variables}
\author{M.C. Fern\'{a}ndez-Fern\'{a}ndez and F.J. Castro-Jim\'{e}nez \thanks{Both
authors partially supported by MTM2007-64509 and FQM333. The first
author is also supported by the FPU Grant AP2005-2360, MEC (Spain).
e.mail addresses: {\tt mcferfer@us.es}, {\tt castro@us.es}}\\
Departamento de \'{A}lgebra \\ Universidad de Sevilla}


\date{20 November 2008}
\maketitle

\begin{abstract}
We describe the Gevrey series solutions at singular points of the
irregular hypergeometric system (GKZ system) associated with an
affine plane monomial curve. We also describe the irregularity
complex of such a  system with respect to its singular support and
in particular we prove, using elementary methods,  that this
irregularity complex is a perverse sheaf as assured by a theorem of
Z. Mebkhout.
\end{abstract}


\section*{Introduction}
To every row integer matrix $A=(a_1\,\, a_2)$, with positive and
relatively prime entries, and every complex parameter $\beta\in \CC$
we can associate the hypergeometric system $\cM_A(\beta)$ defined by
the following two linear partial differential equations:
$$\left(\frac{\partial}{\partial x_1}\right)^{a_2} (\varphi) - \left(\frac{\partial}{\partial
x_2}\right)^{a_1}(\varphi)=0
$$

$$ a_1 x_1 \frac{\partial \varphi}{\partial x_1} + a_2 x_2 \frac{\partial \varphi}{\partial
x_2} -\beta \varphi =0.
$$

General hypergeometric systems have been introduced by I.M.
Gel'fand, M.I. Graev, M.M. Kapranov and A.V. Zelevinsky (\cite{GGZ},
\cite{GZK1}, \cite{GZK}) and their  analytic solutions, at a generic
point in $\CC^n$, have been widely studied (see e.g. \cite{GZK1},
\cite{GZK}, \cite{Adolphson}, \cite{SST},
\cite{ohara-takayama-2007}).

In this work we explicitly describe the Gevrey solutions --at
singular points-- of the hypergeometric system $\cM_A(\beta)$
associated with $A=(a_1\,\, a_2)$ and $\beta\in \CC$. To this end we
will use the $\Gamma$--series introduced in \cite{GZK} and also used
in \cite{SST} in a very useful and slightly different form. We use
these $\Gamma$--series to describe the {\em Gevrey filtration of the
irregularity complex}  of $\cM_A(\beta)$ with respect to the
coordinate axes and in particular with respect to the singular
support of the system.

Despite the simplicity of the equations defining these
hypergeometric systems the study of its Gevrey solutions is quite
involved. In \cite{fernandez-castro-curves-2008} the study of the
Gevrey solutions of the hypergeometric system associated with any
affine monomial curve in $\CC^n$ is reduced to the two dimensional
case by using deep results in $\cD$--module theory. This justifies
our separated treatment for the two variables case.

The behavior of  Gevrey solutions of a hypergeometric system (and
more generally of any holonomic $\cD$--module) is closely related to
its  irregularity  complex as proved by Z. Mebkhout and by Y.
Laurent and Z. Mebkhout (\cite{Mebkhout}, \cite
{Mebkhout_comparison89}, \cite{Laurent-Mebkhout},
\cite{Laurent-Mebkhout2}). For any hypergeometric system in two
variables  we will describe its irregularity complex without using
any of the deep results in the above references. In particular we
will prove (see Conclusions) that the irregularity complex is a
perverse sheaf.

The paper has the following structure. In Section
\ref{Gevrey-series} we recall the definition of Gevrey series. In
Section \ref{GGZ-GKZ-systems} we recall the definition of
$\Gamma$--series and summarize the description of the holomorphic
solutions of $\cM_A(\beta)$ at a generic point of $\CC^2$. Section
\ref{subsub-gevrey-(ab)} is devoted to the definition --due to Z.
Mebkhout \cite{Mebkhout}-- of $\Irr_Y(\cM_A(\beta))$, the
irregularity complex of the system $\cM_A(\beta)$ with respect to
its singular support $Y$ in $\CC^2$. In Section \ref{irr-at-zero} we
prove that the germ of the irregularity complex
$\Irr_Y(\cM_A(\beta))$ at the origin is zero. In Section
\ref{irr-at-p} we first prove that the complex
$\Irr_Y(\cM_A(\beta))_p$ for $p\in Y$, $p\not= (0,0)$, is
concentrated in degree 0 and then we describe a basis of its 0-th
cohomology group. We also prove, by elementary methods, that the
irregularity complex $\Irr_Y(\cM_A(\beta))$ is a {\em perverse
sheaf} on $Y$. This is a very particular case of a theorem of Z.
Mebkhout \cite[Th. 6.3.3]{Mebkhout}.

Some related results can be found in \cite{oaku-on-regular-2007},
\cite{takayama-modified-2007} and also in \cite{majima} and
\cite{iwasaki}.

The second author would like to thank N. Takayama for his very
useful comments concerning logarithm-free hypergeometric series and
for his help, in April 2003,  computing the first example of Gevrey
solutions: the case of the hypergeometric system associated with the
matrix $A=(1\,\, 2)$ (i.e. with the plane curve $x^2-y=0$). The
authors would like to thank L. Narv\'{a}ez-Macarro and J.M. Tornero for
their very useful comments and suggestions.

\section{Gevrey series} \label{Gevrey-series}

Let us write $X=\CC^2$ with its structure of complex manifold,
${\mathcal O}_X $ (or simply $\cO$) the sheaf of holomorphic
functions on $X$ and ${\mathcal D}_X $ (or simply $\cD$) the sheaf
of linear differential operators with coefficients in $\cO_X$. The
sheaf $\cO_X$ has a natural structure of left $\cD_X$--module. Let
$Z$ be a  hypersurface (i.e. a plane curve) in $X$ with defining
ideal $\mathcal{I}_{Z}$. We denote by $\cO_{X|Z}$ the restriction to
$Z$ of the sheaf $\cO_X$ (and  we will also denote by $\cO_{X|Z}$
its extension by 0 on $X$). Recall that the formal completion of
$\cO_X$ along $Z$ is defined as
$$\cO_{\widehat{X|Z}}:= \lim_{\stackrel{\longleftarrow }{k}} \cO_{X} /
\mathcal{I}_{Z}^k.
$$
By definition $\cO_{\widehat{X|Z}}$ is a sheaf on  $X$ supported on
$Z$ and it has a natural structure of left $\cD_X$--module. We will
also denote by $\cO_{\widehat{X|Z}}$ the corresponding sheaf on $Z$.
We denote by $\cQ_Z$ the quotient sheaf defined by the following
exact sequence
$$0\rightarrow \cO_{X|Z} \longrightarrow \cO_{\widehat{X|Z}}
\longrightarrow \cQ_Z \rightarrow 0.$$ The sheaf $\cQ_Z$ has then a
natural structure of left $\cD_X$--module.


Assume $Y\subset X$ is a smooth  curve  and that it is locally
defined by $x_2=0$ for some system of local coordinates $(x_1,x_2)$
around a point $p\in Y$. Let us consider a real number $s\geq 1$. A
germ
$$f=\sum_{i\geq 0} f_i (x_1) x_2^i \in \cO_{\widehat{X|Y},p} $$ is
said to be a Gevrey series of order $s$ (along $Y$ at the point $p$)
if the power series
$$\rho_s (f) := \sum_{i\geq 0} \frac{1}{i!^{s-1}} f_i (x_1) x_2^i $$ is convergent  at $p$.

The sheaf $\cO_{\widehat{X|Y}}$ admits a natural filtration by the
sub-sheaves $\cO_{\widehat{X|Y}}(s)$ of Gevrey series of order $s$,
$1\leq s\leq  \infty$ where by definition
$\cO_{\widehat{X|Y}}(\infty)= \cO_{\widehat{X|Y}}$. So we have
$\cO_{\widehat{X|Y}}(1) = \cO_{{X|Y}}$. We can also consider the
induced filtration on $\cQ_Y$, i.e. the filtration by the
sub-sheaves $\cQ_{Y} (s)$ defined by the exact sequence:
\begin{align}\label{exact-sequence-gevrey-s}
0\rightarrow \cO_{X|Y} \longrightarrow \cO_{\widehat{X|Y}} (s)
\longrightarrow \cQ_Y (s) \rightarrow 0
\end{align}

\begin{definition}\label{def-gevrey-index}
Let $Y$ be a smooth  curve in  $X=\CC^2$ and let $p$ be a point in
$Y$. The Gevrey index of a formal power series $f\in
{\cO_{\widehat{X|Y},p}}$ with respect to $Y$ is the smallest $ 1
\leq s\leq \infty$ such that $f \in {\cO_{\widehat{X|Y}}}(s)_p$.
\end{definition}

\section{The hypergeometric system associated with an affine mo\-no\-mial plane curve}
\label{GGZ-GKZ-systems}

We denote by $A_2(\CC)$ or simply $A_2$ the complex Weyl algebra of
order $2$, i.e. the ring of linear differential operators with
coefficients in the polynomial ring $\CC[x]:=\CC[x_1,x_2]$. The
partial derivative $\frac{\partial}{\partial x_i}$ will be denoted
by $\partial_i$.

Let $A=(a\,\, b) $ be an integer nonzero row matrix and $\beta\in
\CC$. Let us denote by $E_A(\beta)$ the linear differential operator
$E_A(\beta) := a x_1 \partial_1+ b x_2
\partial_2 -\beta$. The toric ideal $I_A \subset
\CC[\partial]:=\CC[\partial_1,\partial_2] $ associated with $A$ is
generated by the binomial $\partial_1^{b'}-\partial_2^{a'}$ where
$a'=a/d$, \, $b'=b/d$ and $d=\gcd(a,b)$.  The algebraic plane curve
defined by $I_A$ is then an affine monomial plane curve.

The left ideal $A_2 I_A + A_2 E_A(\beta) \subset A_2$ is denoted by
$H_A(\beta)$ and it is called the {\em hypergeometric ideal}
associated with $(A,\beta)$. The (global) hypergeometric module
associated with $(A,\beta)$ is by definition (see \cite{GGZ},
\cite{GZK}) the quotient $M_A(\beta):=A_2/H_A(\beta)$.

To the pair $(A,\beta)$ we can also associate the corresponding
analytic hypergeometric $\cD_X$--module, denoted by $\cM_A(\beta)$,
which is the quotient of $\cD_X$ modulo the sheaf of left ideals in
$\cD_X$ generated by $H_A(\beta)$.



In this paper we will assume that $A=(a\,\, b)$ is an integer row
matrix with $0 < a < b$ and $\beta \in \CC$.  We can assume without
loss of generality that $a,b$ are relatively prime. Nevertheless
similar methods to the ones presented here can be applied to
different kind of plane monomial curves (see Remark \ref{-ab}).

The module $\cM_A(\beta)$ is the quotient of $\cD_X$ modulo the
sheaf of ideals generated by the operators
$P:=\partial_1^{b}-\partial_2^a$ and $E_A(\beta)=
ax_1\partial_1+bx_2\partial_2 -\beta$. Sometimes we will write
$E=E(\beta)=E_A(\beta)$ if no confusion is possible.

Although it can be deduced from general results (\cite{GGZ} and
\cite[Th. 3.9]{Adolphson}) a direct computation shows that the
characteristic variety of $\cM_A(\beta)$ is $T^*_X X \cup T^*_{Y} X$
where $Y=(x_2=0)$ and then the singular support of $\cM_A(\beta)$ is
the axis $Y$. The module  $\cM_A(\beta)$ is therefore  holonomic.

\subsection{Holomorphic solutions of $\cM_A(\beta)$ at a generic
point}\label{sol_generic_point}

By \cite[Th. 2]{GZK} and \cite[Cor. 5.21]{Adolphson} the dimension
of the vector space of holomorphic solutions of $\cM_A(\beta)$ at a
point $p\in X\setminus Y$ equals $b$. A basis of such vector space
of  solutions can be described, using $\Gamma$--series (\cite{GGZ},
\cite[Sec. 1]{GZK} and \cite[Sec. 3.4]{SST}) as follows. For $v\in
\CC^2$ and $u\in \ZZ^2$ let us denote
$$\Gamma[v;u]:=\frac{(v)_{u_-}}{(v+u)_{u_+}}$$ if $(v+u)_{u_+}\neq
0$ and $\Gamma[v;u]:=0$ otherwise. Here
$$(z)_\alpha = \prod_{i: \, \alpha_i>0} \prod_{j=0}^{\alpha_i-1}
(z_i-j)$$ is the Pochhammer symbol,  for any $z\in \CC^n$ and any
$\alpha \in \NN^n$.

For $j=0,\ldots,b-1$ let us consider
$$v^j=(j,\frac{\beta -j a}{b})\in \CC^2$$ and the corresponding
$\Gamma$--series $$\phi_{v^j} = x^{v^j} \sum_{m\geq 0} \Gamma[v^j;
u(m)] \left(\frac{x_1^{b}}{x_2^{a}}\right)^m \in
x^{v^j}\CC[[x_1,x_2^{-1}]]$$  with $u(m) = (bm,-am)\in L_A =
\ker_\ZZ(A)$, which defines a holomorphic function at any point
$p\in X\setminus Y$. This can be easily proven by applying
d'Alembert ratio test to the series in $\frac{x_1^{b}}{x_2^{a}}$
$$\psi :=\sum_{m\geq 0} \Gamma[v^j; u(m)]
\left(\frac{x_1^{b}}{x_2^{a}}\right)^m.$$

Writing $c_m:=\Gamma[v^j; u(m)]$ we have
$$\lim_{m\rightarrow \infty}
\left|\frac{c_{m+1}}{c_m}\right| = \lim_{m\rightarrow \infty}
\frac{(am)^a}{(bm)^{b}}=0.$$

Notice that $\phi_{v^j} \not\in \cO_X(X\setminus Y)$ if $\frac{\beta
-j a}{b}\not\in \ZZ$.

\section{Gevrey solutions of
$\cM_A(\beta)$}\label{subsub-gevrey-(ab)}


The definition of the irregularity (or the irregularity complex) of
a left coherent $\cD_X$--module $\cM$ has been given by Z. Mebkhout
\cite[(2.1.2), page 98 and (6.3.7)]{Mebkhout}. In dimension 2 this
definition is the following:


\begin{definition} (a) Let $Z$ be a curve in $X$. The irregularity
of $\mathcal{M}$ along $Z$ (denoted by $\operatorname{Irr}_Z(\cM)$)
is the solution complex of $\cM$ with values in $\cQ_Z$, i.e.
$$\operatorname{Irr}_Z (\mathcal{M}):=\RR \cH om_{\cD_{X}}
(\mathcal{M},\cQ_Z ). $$ (b) If $Y$ is a smooth curve in $X$ and for
each $1\leq s \leq \infty$, the irregularity of order $s$ of $\cM$
with respect to $Y$ is the complex $$\operatorname{Irr}_{Y}^{(s)}
(\cM ):=\R \cH om_{\cD_X }(\cM , \cQ_Y (s)).$$
\end{definition}

Since $\cO_{\widehat{X|Y}}(\infty)= \cO_{\widehat{X|Y}}$ we have
$\operatorname{Irr}_{Y}^{(\infty)} (\cM ) =
\operatorname{Irr}_{Y}^{} (\cM )$.  By definition, the irregularity
of $\cM$ along $Z$ (resp. $\operatorname{Irr}_{Y}^{(s)} (\cM )$) is
a complex in the derived category $D^b(\CC_X)$ and its support is
contained in $Z$ (resp. in $Y$).

In Sections \ref{irr-at-zero} and \ref{irr-at-p} we will describe
the cohomology of the irregularity complex $\Irr_Y(\cM_A(\beta))$,
and moreover we will compute a basis of the vector spaces
$$\mathcal{H}^i (\operatorname{Irr}_{Y}^{(s)}(\mathcal{M}_A
(\beta)))_p = \mathcal{E}xt^i_{\cD_X} ( \mathcal{M}_A (\beta) ,
\cQ_{Y}(s))_p$$ for $p\in Y$, $i\in \NN$ and  $1 \leq s \leq
\infty$.



\begin{lemma}
A free resolution of  $\mathcal{M}_{A}(\beta)$ is given by
\begin{align}
0\longrightarrow \cD \stackrel{\psi_1}{\longrightarrow} \cD^2
\stackrel{\psi_0}{\longrightarrow} \cD
\stackrel{\pi}{\longrightarrow} \mathcal{M}_{A}(\beta )
\longrightarrow 0 \label{resolucionD}
\end{align} where  $\psi_0$ is defined by the column matrix $(P,E)^t$, $\psi_1$ is defined by the row matrix
$(E+ab , -P)$ and  $\pi$ is the canonical projection.
\end{lemma}

\begin{remark}\label{RHomF}
For any left $\cD_X$--module $\cF$ the solution complex $\RR \cH
om_{\cD_X}(\cM_A(\beta),\cF)$ is represented by $$ 0\longrightarrow
\cF \stackrel{\psi_0^*}{\longrightarrow} \cF\oplus \cF
\stackrel{\psi_1^*}{\longrightarrow} \cF \longrightarrow 0$$ where
$\psi_0^*(f) = (P(f),E(f))$ and $\psi_1^*(f_1,f_2) =
(E+ab)(f_1)-P(f_2)$ for $f,f_1,f_2$ local sections in $\cF$.
\end{remark}

\section{Description of
$\Irr_Y(\cM_A(\beta))_{(0,0)}$}\label{irr-at-zero}


\begin{theorem}\label{nulidad_en_origen}
With the previous notations we have
$\Irr_Y^{(s)}(\cM_A(\beta))_{(0,0)}=0$ for all $\beta\in \CC$ and
$1\leq s \leq \infty$. In other words
$$\mathcal{E}xt^i (\mathcal{M}_A (\beta) , \cQ_{Y} (s) )_{(0,0)}=0$$
for all $\beta\in \C$, $1\leq  s \leq \infty$ and  $i\in \N$. Here
the $\cE xt$ groups are taken over the sheaf of rings $\cD_X$.
\end{theorem}

Previous result is related to \cite[Th. 1]{oaku-on-regular-2007}.

Let us denote by $V_A(\beta,s)$ and $W_A(\beta,s)$ the vector spaces
{\footnotesize {$$\left\{\sum_{\alpha \in \NN^2 } a_{\alpha }
x^{\alpha} \in \cO_{\widehat{X|Y}} (s)_{(0,0)}: \, a_{\alpha }= 0
{\mbox { if }} A\alpha =\beta \right\}, \quad \left\{\sum_{\alpha
\in \NN^2 } a_{\alpha } x^{\alpha} \in \cO_{\widehat{X|Y}}
(s)_{(0,0)}: \, a_{\alpha }= 0 {\mbox { if }} A\alpha \not=\beta
\right\}$$}} respectively. Notice that
$V_A(\beta,s)=\cO_{\widehat{X|Y}} (s)_{(0,0)}$ if and only if $\beta
\not\in a\NN + b\NN$.

\begin{lemma}\label{lemaE}
i) The $\CC$--linear map $E_A(\beta): V_A(\beta,s) \rightarrow
V_A(\beta,s)$ is an automorphism  for all $1\leq s\leq \infty$ and
$\beta \in \CC$. In particular, if $\beta \notin a\NN +b\NN$ then
$E_A(\beta)$ is an automorphism of $\cO_{\widehat{X|Y}} (s)_{(0,0)}$
for all $1\leq s\leq \infty$. \\ ii) The $\CC$--linear map $P:
W_A(\beta,s) \rightarrow W_A(\beta-ab,s)$ is surjective  for all
$1\leq s\leq \infty$ and $\beta \in \CC$.
\end{lemma}

\begin{proof} Part {\em ii)} is obvious. Let's prove part {\em i)}.  We have
$E_A(\beta)=ax_1\partial_1+bx_2\partial_2-\beta$ then for
$f=\sum_{\alpha\in \NN^2} f_\alpha x^\alpha \in \CC[[x_1,x_2]]$ we
have $$E_A(\beta)(f)= \sum_{\alpha\in \NN^2} f_\alpha (A\alpha
-\beta)x^\alpha.$$ This implies that $E_A(\beta)$ is an automorphism
of $V_A(\beta ,\infty)$. It is also clear that $E_A(\beta)$ is an
automorphism of $V_A(\beta,1)$. For any $1 < s < \infty$ we have
$\rho_s E_A(\beta)= E_A(\beta)\rho_s$ and then $E_A(\beta)$ is an
automorphism of $V_A(\beta, s)$ (see Section \ref{Gevrey-series} for
the definition  of $\rho_s$). \end{proof}

\begin{corollary}\label{coroE}  $E_A(\beta)$ is an  automorphism of the  vector space $\cQ_Y
(s)_{(0,0)}$ for $1\leq s\leq \infty$ and $\beta \in \CC$.
\end{corollary}

\begin{proof} {\rm {\bf [Theorem \ref{nulidad_en_origen}]}} Let us simply write $E:=E_A(\beta)$.
The complex $\Irr_Y^{(s)}(\cM_A(\beta))_{(0,0)}$ is represented by
the germ at $(0,0)$ of the following complex $$ 0\longrightarrow
\cQ_Y(s) \stackrel{\psi_0^*}{\longrightarrow} \cQ_Y(s)\oplus
\cQ_Y(s) \stackrel{\psi_1^*}{\longrightarrow} \cQ_Y(s)
\longrightarrow 0$$ where $\psi_0^*(f) = (P(f),E(f))$ and
$\psi_1^*(f_1,f_2) = (E+ab)(f_1)-P(f_2)$ for $f,f_1,f_2$ germs  in
$\cQ_Y(s)$ (see Remark \ref{RHomF}). In particular, we only need to
prove the statement for $i=0,1,2$.

For $i=0,2$ the statement follows from Corollary \ref{coroE}
Let us see the case
$i=1$. Let us consider $(\overline{f},\overline{g})\in
\operatorname{Ker} (\psi^{\ast}_1)_{(0,0)}$ (i.e.
$(E+ab)(\overline{f})=P(\overline{g})$). We want to prove that there
exists $\overline{h}\in \cQ_{Y}(s)_{(0,0)}$ such that
$P(\overline{h})=\overline{f}$ and  $E (\overline{h})=\overline{g}$,
where $(\overline{\mbox{\phantom{x}}})$ means modulo
$\cO_{X|Y,(0,0)}=\CC\{x\}$.

From Corollary \ref{coroE} we have that  there exists a unique
$\overline{h}\in \cQ_Y (s)_{(0,0)}$ such that  $E(\overline{h})=
\overline{g}$. Since $P E=(E+ab)P$ and
$(E+ab)(\overline{f})=P(\overline{g})$ we have:
$$(E+ab)(\overline{f})=P(\overline{g})=P (E(\overline{h}))=(E+ab)(P(\overline{h})).$$
Since for all $\beta \in \C$, $E+ab=E_A(\beta-ab)$ is an
automorphism of $\cQ_Y (s)_{(0,0)}$ (see Corollary \ref{coroE}) we
also have $\overline{f}=\overline{P(h)}$. So
$(\overline{f},\overline{g})= (P(\overline{h}), E (\overline{h}) )
\in \operatorname{Im}(\psi_0^{\ast})_{(0,0)}$.
\end{proof}

\begin{remark} From Theorem \ref{nulidad_en_origen} and the long exact
sequence of cohomology associated with the exact sequence
(\ref{exact-sequence-gevrey-s}) we have
$$\cE xt^i(\cM_A(\beta), \cO_{X\vert Y})_{(0,0)} \simeq \cE
xt^i(\cM_A(\beta), \cO_{\widehat{X\vert Y}}(s))_{(0,0)}$$ for $1
\leq s \leq \infty$, $i\in \NN$ and $\beta \in \C$. In fact we have
the following two propositions.
\end{remark}

\begin{proposition} With the previous notations we have $\cE
xt^i(\cM_A(\beta), \cO_{\widehat{X\vert Y}}(s))_{(0,0)} =0$ for all
$\beta\not\in a\NN + b\NN$,  $1\leq s \leq \infty$ and $i\in \NN$.
\end{proposition}

\begin{proof} The proof is similar to the one of Theorem
\ref{nulidad_en_origen} because $E=E_A(\beta)$ is an automorphism of
$\cO_{\widehat{X\vert Y}}(s))_{(0,0)}$.
\end{proof}

\begin{proposition}\label{ext1-origin} With the previous notations we have $$\dim_{\C}(\cE
xt^i(\cM_A(\beta), \cO_{\widehat{X\vert Y}}(s))_{(0,0)})
=\left\{\begin{array}{lcl}
                                                             1 & \mbox{ if } & i=0,1 \\
                                                             0 & \mbox{ if }  & i\geq 2
                                                           \end{array}\right.$$ for
all $\beta \in a\NN + b\NN$ and  $1\leq s \leq \infty$. Moreover,
$\cE xt^i(\cM_A(\beta), \cO_{\widehat{X\vert Y}}(s))_{(0,0)}$ is
generated by a polynomial $\phi_{v^q }$ when $i=0$ and by the class
of $(0,\phi_{v^q })$ when $i=1$ (see the proof of Lemma
\ref{gevrey_index_lemma} for the definition of $\phi_{v^q}$).
\end{proposition}

\begin{proof} By Remark \ref{RHomF} it is enough to
consider $i=0,1,2$. Let's treat first the case $i=2$. Let $h\in
\cO_{\widehat{X\vert Y}}(s)_{(0,0)}$ and write $h=h_1+h_2$ with
$h_1\in V_A(\beta-ab,s)$ and $h_2\in W_A(\beta-ab,s)$. From Lemma
\ref{lemaE} there exist $f\in V_A(\beta-ab,s)$ and $g\in
W_A(\beta,s)$ such that $(E+ab)(f)=h_1$ and $P(g)=-h_2$. Then
$(E+ab)(f)-P(g)=h_1+h_2=h$.

Let's see now that the $\cE xt^0$ has dimension 1. Assume that $h\in
\cO_{\widehat{X\vert Y}}(s)_{(0,0)}$ satisfies $P(h)=E(h)=0$ and
let's write $h=h_1+h_2$ with $h_1\in V_A(\beta,s)$ and $h_2\in
W_A(\beta,s)$. We have $E(h_2)=0$ and then $E(h)=E(h_1)=0$ implies
$h_1=0$ because of Lemma \ref{lemaE}. Now, from $P(h)=P(h_2)=0$ we
get $h_2=\lambda \phi_{v^q}$ for some $\lambda \in \CC$ (see
Proposition \ref{dim_gevrey_a_beta_nongeneric}).

Finally, let's prove that the $\cE xt^1$ has dimension 1. Let's
consider $(f,g)\in \cO_{\widehat{X\vert Y}}(s)_{(0,0)}$ such that
$(E+ab)(f)=P(g)$. Let's write $f=f_1+f_2, g=g_1+g_2$ with $f_1\in
V_A(\beta-ab,s)$, $f_2\in W_A(\beta-ab,s)$, $g_1\in V_A(\beta,s)$
and $g_2\in W_A(\beta,s)$. As $(E+ab)(f_2)=0$ we have
$(E+ab)(f)=(E+ab)(f_1)=P(g_1)+P(g_2)$. This implies $P(g_2)=0$ since
$(E+ab)(f_1)$ and $P(g_1)$ belong to $V_A(\beta-ab,s)$. By Lemma
\ref{lemaE} there exists $h_1\in V_A(\beta,s)$ such that
$E(h_1)=g_1$. We also have $(E+ab)(f_1-P(h_1)) =
(E+ab)(f_1)-PE(h_1)=0$ and again by Lemma \ref{lemaE} we have
$(f_1,g_1)=(P(h_1),E(h_1))$.

By Lemma \ref{lemaE} there exists $h_2\in W_A(\beta,a)$ such that
$P(h_2)=f_2$.  So, $(f_2,g_2)-(P(h_2),E(h_2)) =(0,g_2)=\lambda
(0,\phi_{v^q})$ for some $\lambda \in \CC$ since $P(g_2)=0$  (see
Proposition \ref{dim_gevrey_a_beta_nongeneric}).
\end{proof}


\section{Description of $\Irr_Y(\cM_A(\beta))_p$  for  $p\in Y$,
$p\not=(0,0)$}\label{irr-at-p}

We will compute a basis of the vector space $\mathcal{E}xt^i
(\mathcal{M}_A (\beta ) , \cQ_Y (s))_{p}$ for $1\leq s\leq \infty$,
$i\in \N$, $p\in Y$, $p\not= (0,0)$.  In this section we are writing
$p=(\epsilon,0)\in Y$ with $\epsilon \in \CC^*$.

We are going to use  $\Gamma$--series following (\cite{GGZ},
\cite[Section 1]{GZK}) and in the way they are handled in
\cite[Section 3.4]{SST}.

We will consider the family $v^k=(\frac{\beta-kb}{a},k)\in \CC^2$
for $k=0,\ldots,a-1$. They satisfy $Av^k=\beta$ and the
corresponding $\Gamma$--series are
$$\phi_{v^k} = x^{v^k}\sum_{m\geq 0}{\Gamma[v^k; u(m)]} x_1^{-bm} x_2^{am} \in
x^{v^k} \C[[x_1^{-1} ,x_2 ]]$$ where $u(m)=(-bm,am)$ for $m\in \ZZ$.

Although  $\phi_{v^k}$ does not define in general  any holomorphic
germ at $(0,0)$ it is clear that it defines a germ $\phi_{v^k ,p}$
in $\cO_{\widehat{X|Y},p}$ for $k=0,1,\ldots ,a-1$. Let us write
$x_1=t_1+\epsilon$ and remind that $\epsilon\in \CC^*$. We have
$$\phi_{v^k , p} = (t_1 +\epsilon )^{\frac{\beta -b
k}{a}}x_2^k \sum_{m\geq 0}{\Gamma[v^k;u(m)]}(t_1 +\epsilon )^{-bm}
x_2^{am}.$$

\begin{lemma}\label{gevrey_index_lemma} \begin{enumerate} \item If $\beta \in a\NN +
b\NN$ then there exists a unique $0\leq q \leq a-1$ such that
$\phi_{v^q}$ is a polynomial. Moreover, the Gevrey index of
$\phi_{{v^k},p}\in \cO_{\widehat{X|Y},p}$ is $\frac{b}{a}$ for
$0\leq k \leq a-1$ and $k\not= q$. \item If $\beta \not\in a\NN +
b\NN$ then the Gevrey index of $\phi_{{v^k},p}\in
\cO_{\widehat{X|Y},p}$ is $\frac{b}{a}$ for $0\leq k \leq a-1$.
\end{enumerate}\end{lemma}

\begin{proof} The notion of Gevrey index is given in Definition
\ref{def-gevrey-index}. Let assume first that $\beta\in a\NN +
b\NN$. Then there exists a unique  $0\leq q \leq a-1$ such that
$\beta = qb+a\NN$. Then for $m\in \NN$ big enough $\frac{\beta
-qb}{a}-bm$ is a negative integer and the coefficient ${\Gamma[v^q;
u(m)]}$ is zero.
So $\phi_{v^q}$ is a polynomial in $\CC[x_1,x_2]$ (and then
$\phi_{v^q,p}(t_1,x_2)$ is a polynomial in $\CC[t_1,x_2]$) since for
$\frac{\beta -qb}{a}-bm\geq 0$ the expression
$$x^{v^q}x_1^{-bm}x_2^{am}$$ is a monomial in $\CC[x_1,x_2]$.

Let us consider an integer number $k$ with $0\leq k \leq a-1$.
Assume  $\frac{\beta-bk}{a}\not\in \NN $.  Then the formal power
series $\phi_{v^k,p}(t_1,x_2)$ is not a polynomial. We will see that
its Gevrey index is $b/a$. It is enough to prove that the Gevrey
index of
$$\psi(t_1,x_2):= \sum_{m\geq 0}{\Gamma[v^k;u(m)]}(t_1 +\epsilon
)^{-bm} x_2^{am} =
\\
\sum_{m\geq 0}\Gamma[v^k;u(m)]\left(\frac{x_2^a}{(t_1 +\epsilon
)^{b}}\right)^m $$ is $b/a$, i.e. the series
$$\rho_s(\psi(t_1,x_2))= \sum_{m\geq 0}\frac{\Gamma[v^k;u(m)]} {(am)!^{s-1}}
\left(\frac{x_2^a}{(t_1 +\epsilon )^{b}}\right)^m$$ is convergent
for $s=b/a$ and divergent for $s<b/a$.

Considering $\rho_s(\psi(t_1,x_2))$ as a power series in
$(x_2^a/(t_1+\epsilon)^b)$ and writing $$c_m := \frac{\Gamma[v^k;
u(m)]}{(am)!^{s-1}}$$ we have that
 $$\lim_{m\rightarrow \infty}
\left|\frac{c_{m+1}}{c_m}\right| = \lim_{m\rightarrow \infty}
\frac{(bm)^b}{(am)^{as}}$$ and then by using the d'Alembert's ratio
test it follows that the power series $\rho_s(\psi(t_1,x_2))$ is
convergent for $b\leq as$ and divergent for $b>as$.

\end{proof}

\begin{proposition}\label{dim_formal_a}
We have $\dim_{\CC} \left(\mathcal{E}xt^0 ( \mathcal{M}_A (\beta ) ,
\cO_{\widehat{X|Y}})_{p}\right)=a$ for all $\beta \in \CC$, $p\in Y
\setminus \{(0,0)\}$.
\end{proposition}

\begin{proof} Recall that $p=(\epsilon,0)$ with $\epsilon\in
\CC^*$.  The operators defining  $\mathcal{M}_A (\beta )_{p}$ are
(using coordinates $(t_1,x_2)$) $P=\partial_1^b-\partial_2^a$ and
$E_p(\beta):=at_1\partial_1+bx_2\partial_2+a\epsilon
\partial_1 -\beta$. We will simply write $E_p=E_p(\beta)$.

First of all, we will prove the inequality $$\dim_{\CC}
\left(\mathcal{E}xt^0 ( \mathcal{M}_A (\beta ) ,
\cO_{\widehat{X|Y}})_{p}\right)\leq a.$$ Assume that $f\in \C [[
t_1, x_2 ]]$, $f\neq 0$, satisfies $E_{p}(f)=P(f)=0$. Then choosing
$\omega \in \R_{>0}^2$ such that  $a \omega_2 > b\omega_1 $, we have
$\inww (E_{p})=a \epsilon
\partial_1$ and  $\inww (P)=\partial_2^a$, where $\inww(-)$ stands for the initial part with respect
to the weights $weight(x_i)=-w_i$, $weight(\partial_i)=w_i$.

Then (see \cite[Th. 2.5.5]{SST}) $\partial_1 (\inw (f))=\partial_2^a
(\inw (f))=0$. So,  $\inw (f)= \lambda_l x_2^l$ for  some $0\leq l
\leq a-1$ and some $\lambda_l\in \CC$. This implies the inequality.

Now, remind that
$$\phi_{v^k , p} = (t_1 +\epsilon )^{\frac{\beta -b
k}{a}}x_2^k \sum_{m\geq 0}{\Gamma[v^k; u(m)]}(t_1 +\epsilon )^{-bm}
x_2^{am}$$ and that the support of such a formal series in
$\CC[[t_1,x_2]]$ is contained in $\N \times (k+a\N)$ for
$k=0,1,\ldots ,a-1$. Then the family $\{\phi_{v^k,p}\,\vert \,
k=0,\ldots,a-1\}$ is $\CC$-linearly independent and they all satisfy
the equations defining $\cM_A(\beta)_{p}$.
\end{proof}

\begin{proposition}\label{dim_gevrey_a_beta_gen}
If  $\beta\notin a\N + b\N$ then
$$\mathcal{E}xt^0 ( \mathcal{M}_A (\beta ) ,
\cO_{\widehat{X|Y}}(s) )_{p}= \left\{ \begin{array}{lc}
  \sum_{k=0}^{a-1} \CC \phi_{v^k,p} & \mbox{ if } s\geq \frac{b}{a} \\
  0 & \mbox{ if } s < \frac{b}{a}
\end{array} \right. $$ for all $p=(\epsilon,0) \in \C^{\ast}\times \{0\}$.
\end{proposition}

\begin{proof} From the proof of Proposition \ref{dim_formal_a} and Lemma \ref{gevrey_index_lemma} it
follows that any linear combination $\sum_{k=0}^{a-1} \lambda_k
\phi_{v^k , p}$ with $\lambda_k \in \C$ has Gevrey index equal to
$b/a$ if  $\beta\notin a\N +b\N$.
\end{proof}

\begin{proposition}\label{dim_gevrey_a_beta_nongeneric}
If $\beta\in a\N+b\N$ then
$$\mathcal{E}xt^0 ( \mathcal{M}_A (\beta ),
\cO_{\widehat{X|Y}}(s) )_{p}= \left\{ \begin{array}{lc}
  \sum_{k=0}^{a-1} \C \phi_{v^k,p } & \mbox{ if } s\geq \frac{b}{a} \\
  \C \phi_{v^{q}} & \mbox{ if } s < \frac{b}{a}
\end{array} \right.$$ for all $p=(\epsilon,0) \in \C^{\ast}\times \{0\}$  where  $q$ is the unique   $k\in \{0,1,\ldots
,a-1\}$ such that $\beta \in kb+a\N$.
\end{proposition}

\begin{proof} The proof is analogous to the one of Proposition
\ref{dim_gevrey_a_beta_gen} and follows from Lemma
\ref{gevrey_index_lemma}.
\end{proof}

\begin{lemma}\label{Eepsilonsobre}
The germ of  $E:=E_A(\beta)$ at any point $p=(\epsilon,0)\in
\C^{\ast}\times \{0\}$ induces a surjective endomorphism on
$\cO_{\widehat{X|Y}} (s)_{p}$ for all $\beta \in \C$, $1\leq s\leq
\infty$.
\end{lemma}

\begin{proof}
We will prove that  $E_p:  \cO_{\widehat{X|Y}} (s)_{(0,0)}
\longrightarrow \cO_{\widehat{X|Y}} (s)_{(0,0)}$ is surjective
(using coordinates $(t_1 ,x_2)$. It is enough to prove that
$F:=\partial_1+bx_2u(t_1)\partial_2 -\beta u(t_1)$ yields  a
surjective endomorphism  on $\cO_{\widehat{X|Y}} (s)_{(0,0)}$, where
$u(t_1)=(a(t_1+\epsilon))^{-1} \in\CC\{t_1\}$. For $s=1$, the
surjectivity of $F$ follows from Cauchy-Kovalevskaya theorem. To
finish the proof it is enough to notice that $\rho_s\circ F = F
\circ \rho_s$ for $1\leq s< \infty$. For $s=\infty$ the result is
obvious.
\end{proof}

\begin{corollary}\label{Ext2} We have: \begin{enumerate} \item[i)]
$\mathcal{E} xt^2 (\mathcal{M}_A(\beta),\cO_{\widehat{X|Y}}(s))_{p}
= 0$ for all $p\in Y$, $p\not=(0,0)$, $\beta \in \C$ and $1\leq s
\leq \infty$.
\item[ii)] $\mathcal{E} xt^2(\mathcal{M}_A (\beta ),\cQ_{Y}(s))_{p}=0$ for all $p\in Y$,
$p\not=(0,0)$, $\beta \in \C$ and $1\leq s \leq \infty$.
\end{enumerate}
\end{corollary}

\begin{proof} {\em i)} We first consider the germ at $p$ of the
solution complex of $\cM_A(\beta)$ as described in Remark
\ref{RHomF} for $\cF=\cO_{\widehat{X|Y}}(s)$. Then we apply that $E
+ ab$ is surjective on $\cO_{\widehat{X|Y}}(s)_p$ (Lemma
\ref{Eepsilonsobre}). \\ {\em ii)} It follows from {\em i)} and the
long exact sequence in cohomology associated with
(\ref{exact-sequence-gevrey-s}).
\end{proof}

%

\subsection{Computation of $\mathcal{E}xt^0 ( \mathcal{M}_A (\beta )
, \cQ_{Y} (s))_{p}$ for $p\in Y$, $p\not=(0,0)$
}\label{ext0_Qs_epsilon}

\begin{lemma}\label{formaf}
Assume that $f\in \C [[t_1 ,x_2 ]]$ satisfies  $E_p(f)=0$. Then
$f=\sum_{k=0}^{a-1} f^{(k)}$ where
$$f^{(k)} =\sum_{m\geq 0} f_{k+am}(t_1 +\epsilon
)^{\frac{\beta -b k}{a}-bm} x_2^{k+am}$$ with  $f_{k+am}\in \C$.
\end{lemma}

\begin{proof} Let us sketch the proof. We know that
$\inww(E_p)(\inw (f))=0$ \cite[Th. 2.5.5]{SST} for all $\omega
=(\omega_1,\omega_2) \in \RR^2_{\geq 0}$.

If $\omega_1>0$ then $\inww(E_p)=a\epsilon \partial_1$ and so, $\inw
(f) \in \CC[[x_2]]$ for all $\omega$ with $\omega_1>0$. On the other
hand, if $w_1=0$ then $\inww(E_p)=E_p$ and in particular
$E_p(\operatorname{in}_{(0,1)}(f))=0$ and  $\inw
(\operatorname{in}_{(0,1)}(f))\in \C [x_2 ]$, for all $\omega \in
\R^2_{>0}$.

There exists a unique $(k,m)$ with  $k\in \{0,\ldots ,a-1\}$ and $
m\in \N$ such that $\operatorname{in}_{(0,1)}(f)= x_2^{am+k} h(t_1
)$ for some $h(t_1 )\in \C [[ t_1 ]] $ with $h(0)\neq 0$.

There exists  $f_{am+k}\in \C^{\ast }$ such that  $t_1$ divides
$$\operatorname{in}_{(0,1)}(f)- f_{am+k}(t_1 +\epsilon
)^{\frac{\beta -b k}{a}-bm} x_2^{k+am}\in \C [[ t_1 ]] x_2^{am+k}.$$
But we have
$$E_p
(\operatorname{in}_{(0,1)}(f)- f_{am+k}(t_1
+\epsilon )^{\frac{\beta -b k}{a}-bm} x_2^{k+am})=0.$$ This implies
that  $\operatorname{in}_{(0,1)}(f)= f_{am+k}(t_1 +\epsilon
)^{\frac{\beta -b k}{a}-bm} x_2^{k+am}$.

We finish by induction by applying the same argument to $f-
\operatorname{in}_{(0,1)}(f)$ since $E_p(f-
\operatorname{in}_{(0,1)}(f))=0.$
\end{proof}

Let's recall that $Y=(x_2=0)\subset X=\CC^2$ and $v^k
=(\frac{\beta-bk}{a},k)$ for $k=0,\ldots,a-1$.

\begin{remark}\label{phi_w}
As in the proof of Lemma \ref{gevrey_index_lemma} if  $\beta\in a\NN
+ b\NN$ then there exists a unique $0\leq q \leq a-1$ such that
$\beta \in  qb+a\NN$. Let us write $m_0=\frac{\beta -qb}{a}$.


The series $\phi_{v^q}$ is in fact a polynomial in $\CC[x_1,x_2]$
since for $m_0-bm\geq 0$ the expression $x^{v^q}x_1^{-bm}x_2^{am}$
is a monomial in $\CC[x_1,x_2]$.

Let us write $m'$  the smallest integer number satisfying $bm'\geq
m_0+1$  and
$$\widetilde{v^q}:=v^q+u(m')=(m_0-bm',q+am').
$$ Let us notice that $A \widetilde{v^q} = \beta$ and that $\widetilde{v^q}$
does not have minimal negative support (see \cite[p. 132-133]{SST})
and then the $\Gamma$--series $\phi_{\widetilde{v^q}}$ is not a
solution of $H_A(\beta)$. We have
$$\phi_{\widetilde{v^q}} = x^{\widetilde{v^q}}
\sum_{m\in \NN;\, bm\geq m_0+1} \Gamma[\widetilde{v^q}; u(m)]
x_1^{-bm}x_2^{am}.$$ It is easy to prove that
$H_A(\beta)_p(\phi_{\widetilde{v^q},p}) \subset \cO_{X,p}$ for all
$p=(\epsilon,0)\in X$ with $\epsilon\neq 0$, and that
$\phi_{\widetilde{v^q},p}$ is a Gevrey series of index $b/a$.
\end{remark}

\begin{theorem}\label{Ext0cociente}
For all $p\in Y\setminus\{(0,0)\}$ and $ \beta \in \C$ we have
$$\dim_{\C} (\mathcal{E}xt^0\left( \mathcal{M}_A (\beta ) , \cQ_{Y}
(s))_{p}\right)= \left\{\begin{array}{lcl}
a & \mbox{ if } & s\geq b/a \\
& \\
0 & \mbox{ if } & s<b/a
\end{array}
\right.$$ Moreover, we also have
\begin{enumerate}
\item[i)] If $\beta \notin a\N +b\N$ then:
$$\mathcal{E}xt^0 ( \mathcal{M}_A (\beta ) , \cQ_{Y}
(s))_{p}= \sum_{k=0}^{a-1}\C \overline{\phi_{v^k, p}}$$ for all
$s\geq b/a$
\item[ii)] If  $\beta \in a\N +b\N$ then for all $s\geq b/a$ we have :
$$\mathcal{E}xt^0 ( \mathcal{M}_A (\beta ) , \cQ_{Y}
(s))_{p}= \sum_{k=0 ,k\neq q}^{a-1}\C \overline{\phi_{v^k,p}} +\C
\overline{{\phi}_{\widetilde{v^q},p}}$$ with
${\phi}_{\widetilde{v^q}}$ as in Remark \ref{phi_w}.
\end{enumerate}
Here  $\overline {\phi}$ stands for the class modulo $\cO_{X|Y,p}$
of $\phi \in \cO_{\widehat{X|Y}}(s)_p$.
\end{theorem}
\begin{proof} It follows from Propositions
\ref{dim_gevrey_a_beta_gen} and \ref{dim_gevrey_a_beta_nongeneric},
from the proofs of Lemma \ref{gevrey_index_lemma} and Proposition
\ref{dim_formal_a} (by using the long exact sequence in cohomology)
and Theorem \ref{Ext1Gevrey} below.
\end{proof}

\subsection{Computation  of $\mathcal{E} xt^1 (\mathcal{M}_A (\beta )
,\cQ_Y (s))_p$ for $p\in Y$, $p\not=(0,0)$}

\begin{theorem}\label{Ext1Gevrey}
For all  $\beta \in \C$ we have  $$\mathcal{E} xt^1 (\mathcal{M}_A
(\beta ) ,\cO_{\widehat{X|Y}} (s))_{p}=0$$ for all $s\geq b/a$ and
for all $p\in Y$, $p\not=(0,0)$.
\end{theorem}

\begin{proof} We will use the germ at $p$ of the solution
complex of $\cM_A(\beta)$ with values in
$\cF=\cO_{\widehat{X|Y}}(s)$ (see Remark \ref{RHomF}): $$
0\rightarrow \cO_{\widehat{X|Y}}(s)
\stackrel{\psi_0^*}{\longrightarrow} \cO_{\widehat{X|Y}}(s)\oplus
\cO_{\widehat{X|Y}}(s) \stackrel{\psi_1^*}{\longrightarrow}
\cO_{\widehat{X|Y}}(s) \rightarrow 0$$

Let us consider $(f,g)\in (\cO_{\widehat{X|Y}} (s)_{p})^2$ in the
germ at $p$ of $\ker(\psi_1^*)$, i.e. $(E_p+a b )(f)=P(g)$. We want
to prove that there exists $h \in \cO_{\widehat{X|Y}} (s)_{p}$ such
that $P(h)=f$ and $E_p(h)=g$.

From Lemma  \ref{Eepsilonsobre}, there exists $\widehat{h} \in
\cO_{\widehat{X|Y}} (s)_{p}$ such that  $E_{p} ( \widehat{h})=g$.
Then:

$$(f,g)=(P(\widehat{h}), E_p(\widehat{h}))+
(\widehat{f},0)$$ where  $\widehat{f}=f-P(\widehat{h})\in
\cO_{\widehat{X|Y}} (s)_{p}$ and $(\widehat{f},0)\in
\operatorname{Ker}(\psi_1^{\ast})$. In order to finish the proof it
is enough to prove that there exists $h \in \cO_{\widehat{X|Y}}
(s)_{p}$ such that $P(h)=\widehat{f}$ and $E_p(h)=0$.

Since $h, \widehat{f} \in \C [[t_1 ,x_2 ]]$,
$(E_p+ab)(\widehat{f})=0$ and $E_p(h)$ must be $0$, it follows from
Lemma \ref{formaf} that
$$h=\sum_{k=0}^{a-1} \sum_{m\geq 0} h_{k+am} (t_1 +\epsilon)^{\frac{\beta -b
k}{a}-bm} x_2^{k+am}, $$ $$\widehat{f}=\sum_{k=0}^{a-1} \sum_{m\geq
0} f_{k+am} (t_1 +\epsilon)^{\frac{\beta -b k}{a}-b(m+1)}
x_2^{k+am}$$ with  $h_{k+am}, f_{k+am} \in \C$.

The equation $P(h)=\widehat{f}$ is equivalent to the recurrence
relation:
\begin{align}
h_{k+a(m+1)}=\frac{1}{(k+a(m+1))_a } \left(\left(\frac{\beta -b
k}{a}-bm \right)_b h_{k+am} - f_{k+am}\right) \label{ecrecurrencia}
\end{align}
for $k=0,\ldots, a-1$ and $m\in \NN$. The solution to this
recurrence relation proves that there exists $h\in \C [[ t_1 , x_2
]]$ such that $P(h)=\widehat{f}$ and $E_p (h)=0$.

We need to prove now that  $h\in \cO_{\widehat{X|Y}}(s)_{p}$.

Dividing  (\ref{ecrecurrencia}) by  $((k+a(m+1))!)^{s-1 }$ we get:

$$\frac{h_{k+a(m+1)}}{(k+a(m+1))!^{s-1 }}=$$ $$\frac{1}{((k+a(m+1))_a)^s }
\left(\left(\frac{\beta -b k}{a}-bm \right)_b
\frac{h_{k+am}}{(k+am)!^{s-1 } }- \frac{f_{k+am}}{(k+am)!^{s-1 } }
\right)$$ So it is enough to prove that there exists  $C ,D>0$ such
that
\begin{align}\label{CD} \left|\frac{h_{k+am}}{(k+am)!^{s-1 }}\right|\leq C D^m
\end{align} for all $0\leq k \leq a-1$ and  $m\geq 0$. We will argue by induction on $m$.

Since $\rho_s (\widehat{f})$ is convergent, there exists
$\widetilde{C}, \widetilde{D}>0$ such that
$$\frac{|f_{k+am}|}{(k+am)!^{s-1 } } \leq \widetilde{C}
\widetilde{D}^m$$ for all $m\geq 0$ and $k=0,\ldots,a-1$.


Since  $s\geq b/a$, we have
$$\lim_{m\rightarrow  \infty }\frac{|(\frac{\beta -b k}{a}-bm )_b|}{((k+a(m+1))_a)^s } \leq
(b/a)^b$$ and then there exists an upper bound  $C_1>0$ of the set
$$\left\{ \frac{\left|(\frac{\beta -b k}{a}-bm )_b\right|}{((k+a(m+1))_a)^s } : m\in
\N\right\}.$$ Let us consider  $$C = \max \{\widetilde{C} ,
\frac{|h_k| }{k!^{s-1}}\,; k=0,\ldots,a-1 \}$$ and  $$D= \max\{
\widetilde{D}, C_1 + 1 \}.$$

So, the case  $m=0$ of (\ref{CD}) follows from the definition of
$C$. Assume  $|\frac{h_{k+am}}{(k+am)!^{s-1 }}|\leq C D^m$. We will
prove inequality (\ref{CD}) for $m+1$. From the recurrence relation
we deduce:

$$\left|\frac{h_{k+a(m+1)}}{(k+a(m+1))!^{s-1 }}\right|\leq C_1
\left|\frac{h_{k+am}}{(k+am)!^{s-1 }}\right| +  \widetilde{C}
\widetilde{D}^m$$ and using the induction hypothesis and the
definition of $C,D$ we get:
$$\left|\frac{h_{k+a(m+1)}}{(k+a(m+1))!^{s-1 }}\right|\leq (C_1 +1 ) C D^m \leq C D^{m+1}.$$
In particular  $\rho_s (h)$ converges and  $h\in
\cO_{\widehat{X|Y}}(s)_{p}$.
\end{proof}

\begin{lemma}\label{lemma-ext1}
Assume that $h \in \cO_{\widehat{X|Y},p}$, $p\in Y$, $p\not=(0,0)$,
satisfies  $E(h)=0$ and $P(h)\in \cO_{\widehat{X|Y}}(s)_p$ with $s<
b/a$.  Then: \begin{enumerate}
\item[i)]  If $\beta \not \in  a \N + b\N$ there exists $g \in
\cO_{\widehat{X|Y}}(s)_p$ with $P(h)=P(g)$ and $E(g)=0$.
\item[ii)] If $\beta \in  a \N + b\N$ there exists $g \in
\cO_{\widehat{X|Y}}(s)_p$ with $P(h)=P(g +\lambda_q
\phi_{\widetilde{v^{q}},p})$ and $E(g )=0$.
\end{enumerate}
\end{lemma}

\begin{proof}
Since $E(h)=0$ then $(E+ab)(\widehat{f})=0$ for $\widehat{f}:=P(h)$.
Reasoning as in the proof of Theorem \ref{Ext1Gevrey} we have the
recurrence relation (\ref{ecrecurrencia}) for the coefficients of
$h$ and $\widehat{f}$. Let us prove first that for all $k=0,\ldots
,a-1$ such that $\frac{\beta -bk}{a}\notin \N$ there exists
$\lambda_k \in \C$ with $h^{(k)}-\lambda_k \phi_{v^k ,p}\in
\cO_{\widehat{X|Y}}(s)_p$.

Since $ h_{k} x_1^{\frac{\beta -b k}{a}} x_2^{k}$ is holomorphic in
a neighborhood  of $p$ and $E (h_{k} x_1^{\frac{\beta -b k}{a}}
x_2^{k})=0$ we can assume without loss of generality that $h_k =0$
obtaining:
\begin{align}
h_{k+a(m+1)}=- \frac{(\frac{\beta
-bk}{a})_{b(m+1)}}{(k+a(m+1))!}\sum_{r=0}^{m}
\frac{(k+ar)!}{(\frac{\beta -bk}{a})_{b(r+1)}} f_{k+ar}.
\label{ecexplicita}
\end{align} Recall that the coefficient of $x_1^{\frac{\beta
-b k}{a}-bm} x_2^{k+am}$ in $\phi_{v^k ,p}$ is $$\Gamma
[v^{k};u(m)]= \frac{(\frac{\beta -b k}{a})_{bm} k!}{(k+am)! }.$$
Therefore  for all $\lambda_k \in \C$ we get:
$$h_{k+a(m+1)}-\lambda_k \Gamma [v^{k};u(m +1)]=\frac{(\frac{\beta
-bk}{a})_{b(m+1)}}{(k+a(m+1))!} \left(- k! \lambda_k - \sum_{r=0}^m
\frac{(k+ar)! f_{k+ar}}{(\frac{\beta -bk}{a})_{b(r+1)}}\right).$$
Since  $as<b$ we can choose $$\lambda_k = - \sum_{r\geq 0}
\frac{(k+ar)! f_{k+ar}}{k! (\frac{\beta -bk}{a})_{b(r+1)}} \in \C$$
and, because $f^{(k)}\in \cO_{\widehat{X|Y}}(s)_p$,  there exist
real numbers $C>0, D>0 $ such that $|f_{k+ar}|\leq C D^r (k+ a
r)!^{s-1}$ for all $ r\geq 0$.  Then:
$$h_{k+a(m+1)}-\lambda_k \Gamma [v^{k};u(m +1)]=\frac{(\frac{\beta
-bk}{a})_{b(m+1)}}{(k+a(m+1))!} \sum_{r\geq m+1} \frac{(k+ar)!
f_{k+ar}}{(\frac{\beta -bk}{a})_{b(r+1)}}.$$ Equivalently, $$
h_{k+a(m+1)}-\lambda_k \Gamma [v^{k};u(m +1)] =\sum_{r\geq 0}
\frac{(k+a(r+m+1))_{ar} f_{k+a(r+m+1)}}{(\frac{\beta
-bk}{a}-(m+1)b)_{b(r+1)}}.
$$
%
%
%
%
%
%
The series $$g_m (z)=\sum_{r\geq 0} \frac{(k+a(r+m+1))_{ar}^{s}
}{|(\frac{\beta -bk}{a}-(m+1)b)_{b(r+1)}|} z^r$$ is an entire
function in the variable $z$ for all $m\geq 0$. To prove that it is
enough to apply the d'Alembert's ratio test using $b>s a$:
$$\lim_{r\rightarrow  \infty} \frac{(k+a(r+m+1))^s (k+a(r+m+1)-1)^s \cdots (k+
a (r+m)+1)^s }{|\frac{\beta -bk}{a}-b(r+m+1)|\cdots |\frac{\beta
-bk}{a}-(r+m+2)b +1|}=$$ $$ \lim_{r\rightarrow  \infty} \frac{(a
r)^{a s}}{(b r)^b}=0.$$ In particular, $0< g_m (D)<\infty$ and
$$|h_{k+a(m+1)}-\lambda_k \Gamma [v^{k};u(m +1)]|\leq  C  g_m (D) D^{m+1} (k+ a m)!^{s-1}.$$

It can be proved (by using elementary properties of the Pochhammer
symbol and standard estimates) that there exists $m_{2} ,
\widehat{C} \in \N$ such that $\forall m\geq m_{2}$, $ g_{m+1}
(D)\leq \widehat{C} g_m (D)$. This implies that $|g_{m+1}(D)|\leq
\widehat{C}^{m+1-m_2 } g_{m_2 } (D)$, $\forall m \geq m_2$. Then,
taking $\widetilde{C}=\widehat{C}^{-m_2 } g_{m_2 } (D)
>0$, we obtain:
$$|h_{k+a(m+1)}-\lambda_k \Gamma [v^{k};u(m +1)]|\leq \widetilde{C}
(\widehat{C}D)^{m+1} (k+ a m)!^{s-1}$$ for all $m\geq 0$. Hence
$h^{(k)} - \lambda_k \phi_{v^k ,p}\in \cO_{\widehat{X|Y}}(s)_p$.
If $\beta \notin a\N +b\N$ then we have that $g :=h-\sum_{k=0}^{a-1}
\lambda_k \phi_{v^k ,p} \in \cO_{X|Y}(s)_p$ satisfies statement {\em
i)}.

{\em ii)} If $\beta\in a\N+b\N$ there exists a unique $q\in
\{0,1,\ldots ,a-1\}$ verifying $\frac{\beta-bq}{a}=m_0 \in \N$. For
$k\neq q$ we have as before that $h^{(k)} - \lambda_k \phi_{v^k ,p}
\in \cO_{\widehat{X|Y}}(s)_p$. For $k=q$ we can assume without loss
of generality that $h_{q+am}=0$ for $m=0,1,\ldots ,[m_0 /b]$ (here
$[-]$ denotes the integer part) obtaining an expression for $h_{q+a
(m+1)}$ similar to Equality (\ref{ecexplicita}) for $m\geq [m_0
/b]$. Then by using $\phi_{\widetilde{v^{q}}}$ instead of
$\phi_{v^q}$ we get, alike in {\em i)}, that $h^{}-\lambda_q
\phi_{\widetilde{v^{q}}} \in \cO_{\widehat{X|Y}}(s)_p$.  Hence $ g
:=h-\sum_{k=0, k\neq q}^{a-1} \lambda_k \phi_{v^k ,p}- \lambda_q
\phi_{\widetilde{v^{q}}} \in \cO_{\widehat{X|Y}}(s)_p$ satisfies
{\em ii)}.
\end{proof}

\begin{theorem}\label{Ext1Gevrey2}
We have $$\dim_\CC (\cE xt^1 (\cM_A(\beta),
\cO_{\widehat{X|Y}}(s))_p )= \left\{ \begin{array}{lcl} 1 &
{\mbox{for}} & \beta\in a\NN + b\NN
\\ 0 & {\mbox{for}} & \beta\not\in a\NN + b\NN
\end{array} \right.$$ for all $p\in Y$, $p\not= (0,0)$ and $1 \leq s
< \frac{b}{a}$. Moreover, if  $\beta \in a\N +b \N$ then $\cE xt^1
(\cM_A(\beta), \cO_{\widehat{X|Y}}(s))_p$ is generated by the class
of $(P(\phi_{\widetilde{v^{q}},p}),0)$.
\end{theorem}

\begin{proof}
By definition $$\mathcal{E}xt^1 ( \mathcal{M}_A (\beta ) ,
\cO_{\widehat{X|Y}} (s))_{p}=\frac{\{ (f,g)\in (\cO_{\widehat{X|Y}}
(s)_{p})^2 : \; (E+ab)(f)=P(g) \}}{\{ (P(h),E(h)):\; h\in
\cO_{\widehat{X|Y}}(s)_{p}\} }$$ As in the  proof of Theorem
\ref{Ext1Gevrey} we can assume $g=0$ and then $(E+ab)(f)=0$. This
implies that $f=\sum_{k=0}^{a-1} f^{(k)}$ (see Lemma \ref{formaf})
with
$$f^{(k)}=\sum_{m\geq 0} f_{k+ a m} x_1^{\frac{\beta -b k}{a}-b (m +1 )} x_2^{k+a m} $$
We can then consider  $h\in \cO_{\widehat{X|Y},p}$ as in
(\ref{ecexplicita}) such  that $P(h)=f $ and $E(h)=0$ and apply
Lemma \ref{lemma-ext1}. Furthermore, it is easy to prove that
$P(\phi_{\widetilde{v^{q}}})$ is a Laurent monomial term with pole
along $\{x_1 =0\}$ and hence holomorphic  at any point $p\in
Y\setminus\{(0,0)\}$. This finishes the proof.
\end{proof}

\begin{remark}
Notice  that the generator $(P(\phi_{\widetilde{v^q}}),0)$ does not
define a germ at the origin although the dimension of
$\mathcal{E}xt^1 ( \mathcal{M}_A (\beta ) , \cO_{\widehat{X|Y}}
(s))_{(0,0)}$ is one (see Proposition \ref{ext1-origin}).
Nevertheless, it can be checked that the class of $(0,\phi_{v^q })$
is a generator of $\mathcal{E}xt^1 ( \mathcal{M}_A (\beta ) ,
\cO_{\widehat{X|Y}} (s))_{p}$ at any point of $p\in Y$.
\end{remark}

\begin{proposition} \label{ext1(QY(s))iszero}
For all  $\beta \in \C$ we have $$\mathcal{E} xt^1 (\mathcal{M}_A
(\beta ) ,\cQ_Y (s))=0$$ for all $1\leq s\leq \infty$.
\end{proposition}
\begin{proof}
Since $\mathcal{E} xt^1 (\mathcal{M}_A (\beta ) ,\cQ_Y (s))_{p}=0$
for $p=(0,0)$ (see Section \ref{irr-at-zero}) it is enough to prove
the equality for all $p\in Y\setminus \{(0,0)\}$.

From Corollary  \ref{Ext2} (for $s=1$), Theorem \ref{Ext1Gevrey} and
the long exact sequence in cohomology we get the equality for $s\geq
b/a$. Using again Corollary \ref{Ext2} (for $s=1$), Theorem
\ref{Ext1Gevrey2}, Theorem \ref{Ext0cociente} (only necessary in the
case $\beta \in a\NN +b\NN$) and the long exact sequence in
cohomology we get the equality for $1\leq s <b/a$.
\end{proof}

\begin{remark} \label{-ab} We can also prove,  with similar methods
to the ones presented in this paper, that the irregularity complex
$\Irr_Z(\cM_A(\beta))$ is zero for any $\beta \in \CC$ and for
$A=(-a\,\, b)$ with $a,b$ strictly positive integer numbers and
${\rm gcd}(a,b)=1$. Here $Z$ is either $x_1=0$ or $x_2=0$. Let us
notice that the characteristic variety of $\cM_A(\beta)$ is defined
by the ideal $(\xi_1\xi_2, -ax_1\xi_1+bx_2\xi_2)$ and then its
singular support is the union of the two coordinates axes in
$\CC^2$.
\end{remark}

\section*{Conclusions}

1) In Sections \ref{irr-at-zero} and \ref{irr-at-p} we have proved
that the irregularity  complex $\Irr_Y^{(s)}(\cM_A(\beta))$  is zero
for $1\leq s < b/a$  and is  concentrated in degree 0 for $b/a\leq
s\leq \infty$ (see Theorems  \ref{nulidad_en_origen} and
\ref{Ext0cociente} and Proposition \ref{ext1(QY(s))iszero}).
Moreover, the description of a basis of $\cE
xt^0_{\cD_X}(\cM_A(\beta),\cQ_Y(s))_p$ for $p\in Y$, $p\not= (0,0)$,
and $b/a\leq s \leq \infty$  (see Theorem \ref{Ext0cociente}) proves
that the cohomology of the complex $\Irr_Y^{(s)}(\cM_A(\beta))$ is
{\em constructible}  on $Y$, with respect to the stratification
given by $\{\{(0,0)\}, Y\setminus \{(0,0)\}\}$. This can be also
deduced from a Theorem of M. Kashiwara \cite{kashiwara-overdet-75}.
From the form of the basis it is also easy to see that the
eigenvalues of the corresponding monodromy are simply
$\exp(\frac{2\pi i (\beta - bk)}{a})$ for $k=0,\ldots,a-1$. Notice
that for $\beta\in \ZZ$ one eigenvalue (the one corresponding to the
unique $k=0,\ldots,a-1$ such that $\frac{\beta - bk}{a}\in \ZZ$) is
just 1. See Subsection \ref{ext0_Qs_epsilon} for notations.

2) From the previous results we can also  prove that the complex
$\Irr_Y^{(s)}(\cM_A(\beta))$ is a {\em perverse sheaf} on $Y$ for
any $1\leq s \leq \infty$. This is a very particular case of a
general result of Z. Mebkhout \cite[Th. 6.3.3]{Mebkhout}. To this
end, as $\Irr_Y^{(s)}(\cM_A(\beta))$ is concentrated in degree 0, it
is enough to prove the {\em co-support} condition, which is
equivalent (see \cite{BBD}) to prove that the hypercohomology
$\cH^0_{\{p\}}(\Irr_Y^{(s)}(\cM_A(\beta)))$ with support on $\{p\}$
is zero, for $p\in Y$. This is obvious because $\cE
xt^0_{\cD_X}(\cM_A(\beta), \cQ_Y(s))$ has no sections supported on
points.

3) We have also proved (see Theorem \ref{Ext0cociente} and
Proposition \ref{ext1(QY(s))iszero}) that the {\em Gevrey
filtration} $\Irr^{(s)}_Y(\cM_A(\beta))$ has a unique gap for
$s=b/a$. So the only {\em analytic slope} of $\cM_A(\beta)$ with
respect to $Y$ is $b/a$ \cite[D\'{e}f. 6.3.7]{Mebkhout}. On the other
hand it is also known (see \cite[Th. 3.3]{hartillo_trans} and
\cite{Castro-Takayama}) that the only {\em algebraic slope} of
$\cM_A(\beta)$ is also $b/a$. This fact is a very particular case of
the slope comparison theorem of Y. Laurent and Z. Mebkhout \cite[Th.
2.4.2]{Laurent-Mebkhout}.

\end{document}